\documentclass[11pt, a4paper, twoside]{article}
\usepackage{titling}
\usepackage{geometry}
\usepackage{here}
\usepackage{fancyhdr}

\renewcommand{\thefootnote}{\ensuremath{\fnsymbol{footnote}}}

\usepackage{hyperref}
\hypersetup{
    colorlinks=true,
    linkcolor=blue,
    citecolor=blue,
    filecolor=blue,
    urlcolor=blue,
}

\usepackage{cite}
\usepackage{enumitem}

\usepackage{amsthm, amsmath}
\theoremstyle{plain}
\newtheorem{theorem}{Theorem}[section]
\newtheorem{lemma}[theorem]{Lemma}
\newtheorem{proposition}[theorem]{Proposition}

\theoremstyle{definition}
\newtheorem{definition}[theorem]{Definition}

\theoremstyle{remark}
\newtheorem{remark}[theorem]{Remark}

\numberwithin{equation}{section}
\usepackage{amscd}
\usepackage{amssymb, latexsym, amsmath, amscd}
\usepackage{graphicx}
\def\F{{\mathbb F}}
\def \Z {{\mathbb Z}}
\def \R {{\mathbb R}}
\def \C {{\mathbb C}}
\title{Representation Theory of $UT_3(\mathbb{F}_3)$ and its Applications to Equivariant Decomposition in Neural Architectures}
\newcommand\shorttitle{Representations of $UT_3(\F_3)$ and applications}
\author{\large{Bich Van Nguyen$^{\rm a, \dag,\ddag}$, Nguyen Cao Manh Thang $^{\rm b}$}\\[5pt]
	\footnotesize{$^{\rm a}$ Institute for Artificial Intelligence, University of Engineering and Technology, Vietnam National University, Hanoi, Vietnam\\
		$^{\rm b}$Department of Fundamental and Applied Sciences\\
		University of Science and Technology of Hanoi, Vietnam\\[5pt]
		}}
\newcommand\shortname{B.V. Nguyen, C.M.T. Nguyen}
\date{}%Leave it empty

\date{}%Leave it empty

\begin{document}
\noindent\hrulefill\\
\begin{minipage}{2.5cm}
\includegraphics[scale = 0.25]{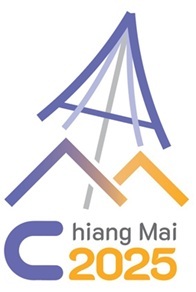}
\end{minipage}
\begin{minipage}{11.5cm}
\begin{center}
The Asian Mathematical Conference (AMC) 2025\\
Department of Mathematics, Faculty of Science\\ 
Chiang Mai University, Chiang Mai, Thailand\\
August 3--7, 2025
\end{center}
\end{minipage}
\hrule height 2pt\hfill\\

\begin{center}
\textbf{\LARGE\thetitle\let\thefootnote\relax}\\[0.5cm]
\theauthor\let\thefootnote\relax\footnotetext{$^\dag$Corresponding author.}\let\thefootnote\relax\footnotetext{$^\ddag$Speaker.}\let\thefootnote\relax\footnotetext{E-mail address: nbvan@vnu.edu.vn (B. V. Nguyen), manhthang190605@gmail.com (C.M.T.Nguyen).}
\vspace{1cm}
\end{center}

\thispagestyle{empty}
%Abstract-----------------------------------------------------
\begin{abstract} 
	In this paper we prove theorems characterizing the decomposition of equivariant feature spaces, filters and a structural preservation theorem for invariant subspace chains in group equivariant convolutional neural networks(G-CNN). Furthermore, we give explicit matrix forms for irreducible representations of $UT_3(\F_3)$-the unitriangular matrix groups over the field with three elements. These results provide a foundation for designing new G-CNN architectures via representations of $UT_3(\F_3)$ that respect deep algebraic structure, with potential applications in symbolic visual learning.

\end{abstract}

\section{Introduction}\label{Sec: Intro}

In recent years, group-equivariant convolutional neural networks (G-CNNs) have emerged as a powerful tool for incorporating symmetry into deep learning models. Originally proposed by Cohen and Welling~\cite{cohen2016}, G-CNNs extend classical convolutional neural networks by replacing translations with actions of more general groups, thereby enhancing learning efficiency and generalization on structured data. 

%Most existing work has focused on abelian or small non-abelian groups such as cyclic groups \( \mathbb{C}_4 \), dihedral groups \( D_8 \), or finite rotation groups in 2D and 3D.

Several successful models, such as SO(2)-CNNs, SE(2)-CNNs, and discrete \( D_8 \)-equivariant networks, have shown significant improvement in image classification, segmentation, and molecular modeling tasks\cite{weiler2019, kondor2018, esteves2020}. More recently, Cohen et al \cite{cohen2019} and Kondor \& Trivedi \cite{kondor2018} developed frameworks for general group convolution over compact Lie groups using harmonic analysis and representation theory. These works highlight the central role of irreducible representations and feature space decompositions in the design of equivariant networks.

However, many real-world data domains exhibit more intricate local symmetries that cannot be captured adequately by the aforementioned groups. One such domain is handwritten mathematical expression recognition, where symbols exhibit local affine distortions, nested substructures, and partial symmetries. These patterns suggest the potential of modeling local transformations using non-abelian nilpotent groups—a class that has so far received little attention in equivariant deep learning.

This paper focuses on the explicit construction and classification of the irreducible complex representations of $\mathrm{UT}_3(\mathbb{F}_3)$, a non-abelian nilpotent group of order 27. We go beyond classification to explore how these representations induce structure-preserving decompositions of equivariant spaces. As a concrete motivation, we connect these ideas with group actions on feature spaces in neural network models, albeit from a purely algebraic perspective without requiring implementation.

%Group equivariant convolutional neural networks (G-CNNs) encode symmetry into neural network layers by utilizing the algebraic action of a group $G$ on feature spaces. A central challenge in their mathematical analysis lies in understanding how group representations decompose the convolutional architecture. In this paper, we address this challenge in the case where $G$ is a finite nilpotent group, and provide a decomposition of feature spaces in terms of irreducible representations.

\section{Preliminaries}\label{Sec: Prelim}

\subsection{Group representations and completely reducibility}
In this subsection we briefly review the necessary concepts from group representation theory which were presented in details in \cite{Serre1977},\cite{Etingof2011} and will serve as the mathematical foundation for our proposed G-CNN model.

\begin{definition}[Group actions]
	Let $G$ be a group and $X$ be a set. An \textbf{action} of $G$ on $X$ is a map $\alpha: G\times X\to X$ such that 
	\begin{enumerate}
		\item $\alpha(e,x)=x\quad \forall x\in X.$
		\item $\alpha(h,\alpha(g,x))=\alpha(hg,x)\quad \forall g,h\in G,\forall x\in X.$
	\end{enumerate}
	We often use the notation $g\circ x$ (or $gx$) instead of $\alpha(g,x)$. Then the above properties can be written as 
	\begin{enumerate}
		\item[1'] $e\circ x=x$
		\item[2'] $h\circ (g\circ x)=(hg)\circ x.$
	\end{enumerate}
\end{definition}
\begin{definition}
	Let $G$ act on 2 sets $X$ and $Y$. Let $\varphi$ be a map from $X$ to $Y$. 
	\begin{enumerate}
		\item $\varphi$ is \textbf{invariant} to the group action if $\varphi(g\circ x)=\varphi (x)\quad \forall g\in G, x\in X.$
		\item $\varphi$ is \textbf{equivariant} to the group action if $\varphi(g\circ x)=g\circ \varphi(x)\quad \forall g\in G,x\in X.$
	\end{enumerate}
	
\end{definition}
The group representation is a mathematical tool to describe a group in
terms of linear transformations of a vector space. It reduces various group-theoretic problems to linear algebra or matrix theory.
\begin{definition}
Let $G$ be a group and $V$ be a vector space over a field $\F$. Let $GL(V)$ be the group of invertible linear transformations of $V$. 

A \textbf{representation} of $G$ on $V$ is a group homomorphism: $\rho: G\to GL(V)$, i.e. $\rho(ab)=\rho(a)\rho(b)\quad\forall a,b\in G$. 

We often write (p, V) to specify a representation and dim(V) is said to be the \textbf{dimension} or \textbf{degree} of this representation. 

 \end{definition}
 \begin{remark}
 		A group representation $\rho: G\to GL(V)$ defines a linear group action of $G$ on $X=V$ by $g\circ x=\rho(g)x.$
 \end{remark}
 \begin{definition}
 	Let $G$ be a group. Let $(\rho,V),(\rho',V')$ be representations of $G$. A linear map $T: V\to V'$ is called a \textbf{$G$-linear} map (or \textbf{$G$-equivariant} map, or \textbf{$G$-map}, or \textbf{an intertwiner}) if $\rho'(g)(T(v))=T(\rho(g)(v)) $ for any $g\in G,v\in V.$
 
 The vector space of $G$-maps between $V$ and $V'$ is denoted by $Hom_G(V,V').$  
 	
 	A $G$-map $T$ is called an \textbf{isomorphism} if it is invertible.
 	
 	Two representations $(\rho,V)$ and $(\rho',V')$ are isomorphic if there is an isomorphism $T: V\to V'$.
 \end{definition}
 \begin{lemma}
 	[Schur's Lemma]\label{schur}
 	Let $(\rho_1,V_1),(\rho_2,V_2)$ be two irreducible representations of a group $G$ and let $T:V_1\to V_2$ be a $G$-map, i.e. $\rho_2(g)\circ T=T\circ \rho_1(g)$ for all $g\in G$. Then:
 	\begin{enumerate}
 		\item If $\rho_1$ and $\rho_2$ are not isomorphic, then $T=0$.
 		\item If $V_1=V_2=V$ and $\rho_1=\rho_2$, then $T=\lambda Id_{V}.$
 	\end{enumerate}
 \end{lemma}
 The proof of this lemma can be found in \cite{Serre1977}. 
 \begin{definition}
 	If a subspace $W\subset V$ of a representation $(\rho,V)$ satisfies $\rho(g)W\subset W$ for any $g\in G$, then the restriction $\rho|_W: G\to GL(W)$ is called a \textbf{subrepresentation} of $(\rho,V)$. 
 	
 	A representation $(\rho,V)$ is \textbf{reducible} if there is a non-trivial
 	subrepresentation (i.e., there is a subspace $W\subset V, W\neq 0,V$ such that $\rho(g)W\subset W$ for any $g\in G.$
 	
 	If a representation is not reducible, then it is called \textbf{irreducible}.
 \end{definition}
	\begin{definition}
	The direct sum of representations $(\rho_1,V_1)$ and $(\rho_2,V_2)$ is defined by $$\rho_1\oplus \rho_2: G\to V_1\oplus V_2, g\mapsto \rho_1(g)\oplus \rho_2(g)$$ and 
	\begin{equation}
		\begin{pmatrix}
			\rho_1(g)&0\\0&\rho_2(g)
		\end{pmatrix}\begin{pmatrix}v_1\\v_2\end{pmatrix}=\begin{pmatrix}
			\rho_1(g)v_1\\\rho_2(g)v_2
		\end{pmatrix}
	\end{equation}
	for all pairs $(v_1,v_2)\in V_1\oplus V_2.$
\end{definition}
	\begin{definition}\label{sem}
	A representation $\rho$ is called \textbf{completely reducible} (or \textbf{semisimple}) if it is
	isomorphic to a direct sum of irreducible representations: $\rho\cong \rho_1\oplus \rho_2\oplus...\oplus \rho_k.$
	
	In terms of matrices, we can make it simultaneously block-diagonal so that
	each block corresponds to an irreducible representation, i.e. there exists an invertible matrix $P$ such that 
	\begin{equation}
		P^{-1}\rho(g)P=\begin{pmatrix}
			\rho_1(g)&0&0&...&0\\
			0&\rho_2(g)&0&...&0\\
			\vdots&\vdots&\vdots&\vdots&\vdots\\
			0&0&0&...&\rho_k(g)
		\end{pmatrix},\quad \forall g\in G
	\end{equation}
	where $P$ does not depend on $g$.
	\end{definition}
	For simplicity we recall only a particular case of Maschke's theorem, namely the theorem for real and complex representations:
	\begin{theorem}\label{maschke}\cite{Serre1977}
		Every finite-dimensional complex (or real) representation of a finite group is completely reducible. 
	\end{theorem}
\subsection{Character theory}
	\begin{definition}
	Let $\rho: G\to GL(V)$ be a finite dimensional representation of a group $G$. For each $g\in G$, let $\chi_\rho(g):=Tr(\rho(g)).$
	
	The function $\chi:G\to \F, g\mapsto \chi_\rho(g)$ is called the \textbf{character} of $\rho$. The character of an irreducible representation of $G$ is called an \textbf{irreducible character} of $G$. 
\end{definition}
The importance of this function comes primarily from the fact that it characterizes the representation $\rho$. 

\begin{proposition}[Properties of characters]\label{char4}\cite{Serre1977}
	
	Let $\rho: G\to GL(V)$ be a complex finite dimensional representation of a finite group $G$. Then:
	\begin{enumerate}
		\item $\chi_\rho(e)=dim(\rho)=dim(V).$
			\item $\chi_\rho(g^{-1})=\overline{\chi_\rho(g)}$ (the complex conjugate).
		\item $\chi_\rho(hgh^{-1})=\chi_\rho(g)$ for all $g,h\in G$. 
	\end{enumerate}
\end{proposition}
It is easy to see that $\chi_{\rho_1\oplus \rho_2}(g)=Tr\begin{pmatrix}
	\rho_1(g)&0\\
	0&\rho_2(g)
\end{pmatrix}=Tr(\rho_1(g))+Tr(\rho_2(g))=\chi_{\rho_1}(g)+\chi_{\rho_2}(g)$ for any $g\in G$. So 
\begin{equation}\label{sum}
	\chi_{\rho_1\oplus \rho_2}=\chi_{\rho_1}+\chi_{\rho_2}.
\end{equation}
	\begin{definition}
	Let $\rho$ and $\rho'$ be 2 complex representations of a finite group $G$. We define the scalar product of $\chi_{\rho}$ and $\chi_{\rho'}$ by
	\begin{equation}
		\langle \chi_{\rho},\chi_{\rho'}\rangle=\frac{1}{|G|}\sum_{g\in G}\chi_{\rho}(g)\overline{\chi_{\rho'}(g)}.
	\end{equation}
\end{definition}

\begin{theorem}\label{ort1}\cite{Serre1977}
	a)	If $\rho$ and $\rho'$ are 2 nonisomorphic irreducible representations of a finite group $G$, then $\langle \chi_{\rho},\chi_{\rho'}\rangle=0.$
	
	b) $\rho$ is an irreducible representation of a finite group $G$ iff $\langle \chi_\rho,\chi_\rho\rangle=1.$
\end{theorem}

From this theorem one deduces that irreducible characters of a finite group form an orthonormal system.

Let $(\rho_i,V_i)$ ($i=\overline{1,r})$ be all the irreducible complex representations (up to isomorphisms, ordered with non decreasing dimensions) of a finite group $G$. 
Suppose that $(\rho,V)$ be a representation of $G$. Following Maschke's Theorem we have the unique decomposition
\begin{equation}
	V=\oplus_{i=1}^rV_i^{\oplus m_i}.
\end{equation}

$(m_1,m_2,...,m_r)$ is called the \textbf{type} of $\rho$.
For each $\rho_j$ we have
\begin{theorem}\label{char}\cite{Serre1977}\begin{equation}
		\langle \chi_\rho,\chi_{\rho_j}\rangle=m_j.
\end{equation}\end{theorem}
\begin{proof}
	Following \eqref{sum} and Theorem \ref{ort1} we have:
	\begin{equation}
		\langle \chi_\rho,\chi_{\rho_j}\rangle=\langle\sum_{i=1}^nm_i\chi_{\rho_i},\chi_{\rho_j}\rangle=\sum_{i=1}^nm_i\langle\chi_{\rho_i},\chi_{\rho_j}\rangle=m_i\delta_{ij}=m_j.
	\end{equation}
\end{proof}
\begin{definition}
	Let $G$ be a group and $g_1,g_2\in G$. We say that $g_2$ is conjugated to $g_1$ if there exists $h\in G$ such that $g_2=h^{-1}g_1h.$
\end{definition}

It is easy to see that the conjugacy relation is an equivalence relation. The equivalence class of element $g$ is called the \textbf{conjugacy class} of $g$.

The following theorem allows us to find the number of nonisomorphic irreducible representations of a group $G$. 
	\begin{theorem}\label{nirr}\cite{Serre1977}
	The number of nonisomorphic irreducible representations of a finite group is equal to the number of its distinct conjugacy classes. 
\end{theorem}
\subsection{Induced representations}
In this subsection we describe how we can construct a representation of a finite group $G$ called the \textbf{induced representaion} from a representation of a subgroup $H$ of $G$. There are several equivalent definitions of induced representations (\cite{Serre1977}, \cite{Etingof2011}). In this paper, for simplicity we choose the following definition:
\begin{definition}\cite{Etingof2011}
%	Let $H\leq G$, $\theta: H\to W$ be a representation which equips $W$ with a structure of $\CH$-left module. Let $V=\C G\otimes_{\C H} W$. The presentation $Ind_H^G(\theta): G\to GL(V)$ given by $Ind_H^G(\theta)(g)(\sigma\otimes w)=g\sigma\otimes w$ is called induced by $\theta$.   
Let $H\leq G.$ Given a representation $\rho_W: H\to GL(W)$ of $H$. 

The induced representation $Ind_H^G(\rho_W)$ is the representation $G\to GL(V)$ where $V=Ind_H^G(W)=\{f:G\to W|f(hx)=\rho_W(h)f(x) \forall x\in G,h\in H\}$ such that $gf(x)=f(xg).$ 
\end{definition}
\begin{remark}\label{ind6}
	Notice that if we choose a representative $x_\sigma$ from every right H-coset $\sigma$ of G, then any $f\in V$ is uniquely determined
	by $\{f(x_\sigma)\}.$ 
	
	So $dim(V)=dim(Ind_H^G(W))=dim(W)[G:H].$
\end{remark}
The following theorem allows us to find the characters of the induced representation:
\begin{theorem}
	Let $\theta: H\to GL(W)$ be a representation of $H\subset G$ and let $R$ be the set of representatives of $G/H$. Then, for each $u\in G$ we have
	\begin{equation}\label{char3}
		Ind_H^G(\chi_\theta)(u)=\chi_{Ind_H^G(\theta)}(u)=\sum_{r\in R,r^{-1}ur\in H}\chi_\theta(r^{-1}ur)=\sum_{s\in G;s^{-1}us\in H}\chi_\theta(s^{-1}us).
	\end{equation} 
\end{theorem}

\subsection{From CNN to G-CNN and Steerable CNNs}
In this framework,  a feature map is a function defined on a domain \( X \), \( f: X \to V \), where \( V \) is a vector space. For classical convolutional neural networks (CNNs), \( X = \mathbb{Z}^2 \) (image pixels), while for G-CNNs, \( X = G \) (group elements). A feature space is a vector space of feature maps and is often denoted by $\mathcal{F}_V$.  
\subsubsection{Convolutional neural networks}
CNNs (\cite{zhang1988}) are deep learning models which are inspired by the biological neural
networks of early visual cortex and based on translation symmetry. The 2D convolution layer $l$ operates on a stack of feature maps $f: \Z^2\to \R^{K^l}$   and $K^{l+1}$ filters $\psi^i: \Z^2\to \R^{K^l}$ , producing an output via:
\begin{equation}
	(f\star \psi^i)(x)=\sum_{y\in \Z^2}\sum_{k=1}^{K^l}f_k(y)\psi^i_k(y-x).
\end{equation}

CNNs were the first neural networks capable of achieving translation-equivariance: when the input image is shifted by $t$, the output feature maps also shift by $t$. 
\begin{multline}
	(L_tf\star \psi^i)(x)=\sum_{y}\sum_kf_k(y-t)\psi^i_k(y-x)\\\stackrel{y':=y-t}{=}\sum_{y'}\sum_kf_k(y')\psi_k^i(y'+t-x)\stackrel{\text{replacing } y' \text{ by } y}{=}\sum_y\sum_kf_k(y)\psi^i_k(y-(x-t))=(L_t(f\star \psi^i))(x).
\end{multline}

A limitation of classical CNNs is that they are only equivariant to translations: if the input shifts, the output shifts accordingly. They are not naturally equivariant to other transformations such as rotations, reflections, or algebraic symmetries.
\subsubsection{Group-equivariant convolutional neural networks}

Group-equivariant CNNs (G-CNNs), introduced by Cohen and Welling (\cite{cohen2016}), generalize the idea of CNNs by replacing the translation  by a more general action of a group $G$ . In the first layer the group convolution is defined by 

\[(f \star \psi)(g) = \sum_{y \in \Z^2}\sum_k f_k(y) \psi_k(g^{-1} y),\] for each $g\in G$.

From this formula, we see that although both the input image $f$ and the filter $\psi$ are functions on $\Z^2$, the result of the group convolution-the feature map $f\star \psi$ is a function on a discrete group $G$ (which may contain $Z^2$ as a subgroup).  Therefore for all layers after the first, the filters $\psi$
must also be functions on $G$, and the convolution operation
becomes
\begin{equation}
(f \star \psi)(g) = \sum_{h \in G}\sum_k f_k(h) \psi_k(g^{-1} h).
\end{equation}
For each $g\in G$, let $L_g$ be the left translation operator defined by
\begin{equation}\label{gcnn1}L_gf(h):=f(g^{-1}h).\end{equation}

It is easy to see that the map $L_g: \mathcal{F}_V\to \mathcal{F}_V$ is invertible and its inverse is $L_{g^{-1}}$. 

Furthermore, the map $L: G\to GL(\mathcal{F}_V),g\mapsto L_g$ is a homomorphism, since $L_{gu}f(h)=f((gu)^{-1}h)=f(u^{-1}g^{-1}h)=L_gL_hf(h).$

Therefore 
\begin{proposition}
The map	$L: G\to GL(\mathcal{F}_V), g\mapsto L_g$ is a representatioon of $G$.
\end{proposition}

The representation $L$ is called the \textbf{regular representation} of $G$. 

\begin{proposition}
	The group convolution layer is equivariant under the action of $G$ by the left-translation:
	\begin{equation}
		(L_uf\star \psi)=L_u(f\star \psi), \quad \forall u\in G.
	\end{equation}
\end{proposition}
\begin{proof} For all $g\in G$ we have
	\begin{multline}
		(L_uf\star \psi)(g)=\sum_{h \in G}\sum_k f_k(u^{-1}h) \psi_k(g^{-1} h)\stackrel{h':=u^{-1}h}{=}\sum_{h'\in G}\sum_kf_k(h')\psi_k(g^{-1}uh')\\\stackrel{\text{replacing } h' \text { by } h}{=}\sum_{h\in G}\sum_kf_k(h)\psi_k((u^{-1}g)^{-1}h)=(L_u(f\star \psi))(g).
	\end{multline}
\end{proof}
This structure guarantees that if the input is transformed by a group element \( g \), the output will transform in a predictable way according to a representation of \( G \).

\subsubsection{Equivariant Convolution: G-CNNs vs. S-CNNs}
\subsection*{Feature Spaces and Filters in S-CNNs}
While G-CNNs extend symmetry to larger groups, they still assume scalar feature maps. Steerable CNNs (S-CNNs), introduced in later works (e.g., \cite{cohen2017} \cite{weiler2019}), further generalize G-CNNs by allowing feature maps and kernels to transform under arbitrary representations of the group. That is, both the signal and the kernel live in representation spaces, and convolutional layers become learnable \emph{intertwiners} between these spaces.
Let \( G \) be a finite group acting on itself by left multiplication. In this paper, the feature space of an S-CNN at a given layer is modeled as the space of functions:
\[
\mathcal{F}_V := \{ f: G \to V \}.
\]
where \( V \) is a complex vector space equipped with a linear representation \( \rho_V: G \to \mathrm{GL}(V) \). The action of the group \( G \) on \( \mathcal{F}_V \)  defined by:
\begin{equation}\label{scnn1}
(\tilde{\rho}(g) f)(h) := \rho_V(g) f(g^{-1} h), \quad \forall g, h \in G.
\end{equation}

This action defines a representation of $G$ in the feature space $\mathcal{F}_V$. 

Remark that in the S-CNN architecture in the papers \cite{cohen2017}, \cite{weiler2019} the feature spaces are constructed by inducing representations from a subgroup $ H \subset G $ to the full group $ G $. This is particularly necessary when $ G $ is a continuous group (e.g., $E(2), \mathrm{SE}(2)$) or too large to classify all irreducible representations explicitly.

In the case of $UT_3(\F_3)$, however,  all the irreducible representations are known explicitly, including both 1-dimensional and 3-dimensional complex representations. So there is no need to construct induced representations from a subgroup.

Instead, we can directly assign to each layer a feature space that decomposes into known irreducible representations $ G $. This leads to a cleaner and a simpler implementation of equivariant convolution.

\begin{definition}[Kernel/Filter]
A convolutional filter is a function:
\[
\psi: G \to \mathrm{Hom}(V, W)
\]
where \( W \) is another representation space with \( \rho_W: G \to \mathrm{GL}(W) \). The filter defines how to linearly combine input features to produce output features at each group location. 

To ensure equivariance, the filter must satisfy the constraint:
\begin{equation}\label{scnn2}
	\psi(x)\rho_V(u) = \rho_W(u) \psi(x)  \quad \forall x, u \in G.
\end{equation}

\end{definition}
\begin{remark}
	The constraint \eqref{scnn2} is equivalent to:
\begin{equation}
	\label{scnn3}	\psi(x)=\rho_W(u) \psi(x)\rho_V(u^{-1}) \quad \forall x, u \in G.
\end{equation}
\end{remark}
The steerable convolution becomes
\begin{equation}
	(f\star \psi)(g):=\sum_{h\in G}\psi(g^{-1}h)f(h).
\end{equation}
\begin{proposition}
	The convolution defined above is equivariant, i.e. it satisfies
	\begin{equation}
		((\tilde{\rho}_V(u)f)\star \psi)(g)=(\tilde{\rho}_W(u)(f\star \psi))(g)
	\end{equation}
	for all $u,g\in G.$ 
\end{proposition}
\begin{proof}
By definition and condition \eqref{scnn2} we have: 
\begin{multline}
	((\tilde{\rho_V}(u)f)\star \psi)(g)=\sum_{h\in G}\psi(g^{-1}h)\rho_V(u)f(u^{-1}h)\stackrel{h'=u^{-1}h}{=}\sum_{h'\in G}\psi(g^{-1}uh')\rho_V(u)f(h')\\\stackrel{h'\to h}{=}\sum_{h\in G}\psi((u^{-1}g)^{-1}h)\rho_V(u)f(h)=\sum_{h\in G}\rho_W(u)\psi((u^{-1}g)^{-1}h)f(h)=(\tilde{\rho}_W(u)(f\star \psi))(g).
\end{multline}
\end{proof}

  In our work, we adopt the steerable CNN framework using the nilpotent group $UT_3(\F_3)$ and design equivariant layers using its irreducible representations. This allows us to exploit the structure of $UT_3(\F_3)$ and ensures precise symmetry behavior across the network.

\section{Theoretical Results}
\subsection{Decomposition of equivariant feature spaces and filters}

%Let \( G = \mathrm{UT}_3(\mathbb{F}_3) \), the group of upper unitriangular 3×3 matrices over the finite field \( \mathbb{F}_3 \). 
Let $\{(\mu_1,V_1),(\mu_2,V_2),...,(\mu_r,V_r)\}$ is the full list of pairwise non-isomorphic irreducible representations of a finite group $G$. 

Now let \( V \) and \( W \) be finite-dimensional (real or complex ) representations of \( G \), corresponding to the input and output feature spaces of a convolutional layer in a G-CNN.

Since $G$ is finite, by using Maschke's theorem \ref{maschke} and collecting isomorphic irreducible representations of $G$ together we have the decomposition
\begin{equation}\label{dec} V = \bigoplus_{i\in I} V_i^{\oplus m_i}, \quad W = \bigoplus_{j\in J}V_j^{\oplus n_j} \end{equation}
where $I,J\subseteq \{1,2,...,r\}$  and \( m_i, n_j \in \Z^+ \) denote positive multiplicities.
\begin{theorem}\label{dec1}
	The decomposition \eqref{dec} induces the following decomposition of the funtion spaces:
	\[
	\mathcal{F}_V=\{ f: G \to V \}=\bigoplus_i \mathcal{F}_{V_i}^{\oplus m_i}, \quad 
	\mathcal{F}_W=\{ f: G \to W \}=\bigoplus_j \mathcal{F}_{V_j}^{\oplus n_j}.
	\]
\end{theorem}
\begin{proof}
	Suppose that $f\in \mathcal{F}_V.$ For each $g\in G$ since $f(g)\in V$ and $V=\bigoplus_{i=1}^{r} V_i^{\oplus m_i}$, so $f(g)$ can be written as a $\sum_{i\in I}n_i$-tuple of vectors $v_{ik}$ , where $i\in I, k=\overline{1,n_i}, v_{ik}\in V_i$. 
	%in this case we use the external definition of the direct sum of vector spaces

	Now if we define $f_{ik}(g):=v_{ik},$ then $f_{ik}\in \mathcal{F}_{V_i}$ and $f$ can be written as a $\sum_{i\in I}m_i$-tuple of functions $f_{ik}\in \mathcal{F}_{V_i}$. 
	
	Furthermore, for any two functions $f,\phi\in \mathcal{F}_V$ we have $(f+\phi)(g)=(f(g)+\phi(g)=(v_{ik})_{i\in I, k=\overline{1,m_i}}+(w_{ik})_{i\in I, k=\overline{1,m_i}}=(v_{ik}+w_{ik})_{i\in I, k=\overline{1,m_i}}\implies (f+\phi)_{ik}=f_{ik}+\phi_{ik}.$
	
	Similarly, we can verify that $(cf)_{ik}=cf_{ik}$ for any $f\in \mathcal{F}_V, c\in \C.$
	
	Therefore $\mathcal{F}_V=\bigoplus_i \mathcal{F}_{V_i}^{\oplus m_i}.$
\end{proof}
\begin{theorem}[Equivariant filter partition]\label{red}
Let $(\rho_V,V)$ and $(\rho_W,W)$ be two complex representations of a finite group $G$. Let $\psi: G\to Hom(V,W)$ be an equivariant filter. Then  there exist invertible matrices $Q_{in}\in M_{dim(V)}(\C)$ and $Q_{out}\in M_{dim(W)}(\C)$ such that for all $g\in G$ $$\kappa(g)=Q_{out}^{-1}\psi(g)Q_{in}$$ is a block matrix  in which  $\kappa_{ji}(g)=0$  when $i\neq j$ and 
$$\kappa_{ii}(g)=A^{(i)}(g)\otimes I_{d_i},$$ where $A^{(i)}(g)$ is a matrix and $d_i$ is a positive integer which are determined by the decompositions of $V$ and $W$ into the direct sums of irreducible subrepresentations. 
\end{theorem}                                                                                           
\begin{proof}
	Following the Maschke's Theorem \ref{maschke} we have the decompositions \ref{dec} corresponding to the decompositions 
	$\rho_V\cong\bigoplus_{i\in I} \mu_i^{\oplus m_i}, \rho_W\cong \bigoplus_{j\in J}\mu_j^{\oplus n_j}.$ 
	
Let $I=\{i_1,...,i_t\},J=\{j_1,...,j_s\}.$ 
	Now by Definition 
	\ref{sem}, the constraint \eqref{scnn3} there exist invertible matrices $Q_{in}\in M_{dim(V)}(\C),Q_{out}\in M_{dim(W)}(\C)$ such that 
	\begin{equation}\label{dec4}
	\psi(g)=Q_{out}diag(\mu_{j_1}^{\oplus n_{j_1}}(u),...,\mu_{j_s}^{\oplus n_{j_s}}(u))Q_{out}^{-1}\psi(g)Q_{in}diag(\mu_{i_1}^{\oplus m_{i_1}}(u^{-1}),...,\mu_{i_t}^{\oplus n_{i_t}}(u^{-1}))Q_{in}^{-1}
	\end{equation}
	
	Set $\kappa(g)=Q_{out}^{-1}\psi(g)Q_{in}.$ Then from \eqref{dec4} we deduce
	\begin{equation}
 		\kappa(g)=diag(\mu_{j_1}^{\oplus n_{j_1}}(u),...,\mu_{j_s}^{\oplus n_{j_s}}(u))\kappa(g)diag(\mu_{i_1}^{\oplus m_{i_1}}(u^{-1}),...,\mu_{i_t}^{\oplus n_{i_t}}(u^{-1})).
	\end{equation}
	We rewrite this equality in the block-matrix form:
	\begin{multline}
		\begin{pmatrix}
			\kappa_{j_1i_1}(g)&\kappa_{j_1i_2}(g)&...&\kappa_{j_1i_t}(g)\\
			\kappa_{j_2i_1}(g)&\kappa_{j_2i_2}(g)&...&\kappa_{j_2i_t}(g)\\
			...&...&...&...\\
			\kappa_{j_si_1}(g)&\kappa_{j_si_2}(g)&...&\kappa_{j_si_t}(g)\\
		\end{pmatrix}=\begin{pmatrix}
		\mu_{j_1}^{\oplus n_{j_1}}(u)&0&...&0\\
		0&\mu_{j_2}^{\oplus n_{j_2}}(u)&...&0\\
		...&...&...&...\\
		0&0&...&\mu_{j_s}^{\oplus n_{j_s}}(u)
		\end{pmatrix}\times\\\times	\begin{pmatrix}
		\kappa_{j_1i_1}(g)&\kappa_{j_1i_2}(g)&...&\kappa_{j_1i_t}(g)\\
		\kappa_{j_2i_1}(g)&\kappa_{j_2i_2}(g)&...&\kappa_{j_2i_t}(g)\\
		...&...&...&...\\
		\kappa_{j_si_1}(g)&\kappa_{j_si_2}(g)&...&\kappa_{j_si_t}(g)\\
		\end{pmatrix}\times \begin{pmatrix}
		\mu_{i_1}^{\oplus m_{i_1}}(u^{-1})&0&...&0\\
		0&\mu_{i_2}^{\oplus n_{i_2}}(u^{-1})&...&0\\
		...&...&...&...\\
		0&0&...&\mu_{i_t}^{\oplus n_{i_t}}(u^{-1})
		\end{pmatrix}.
	\end{multline}
	Performing the multiplication of block matrices gives:
	\begin{equation}
		\kappa_{ji}(g)=\mu_j^{\oplus n_j}(u)\kappa_{ji}(g)\mu_i^{\oplus m_i}(u^{-1}).
	\end{equation}
	If we continue to consider $\kappa_{ji}(g)$ as a block matrix with blocks $\kappa_{ji,kl}$, where $k\in \{1,...,n_j\},l\in \{1,...,m_i\}$, then we get 
	\begin{equation}
		\kappa_{ji,kl}(g)=\mu_j(u)\kappa_{ji,kl}(g)\mu_i(u^{-1}).
	\end{equation}
	Therefore $\kappa_{ji,kl}(g)\in Hom_G(V_i,V_j).$  By Schur's Lemma \ref{schur} we have
	\begin{equation}\label{dec5}
		\kappa_{ji,kl}(g)=\begin{cases}
			0 \text{ if } j\neq i,\\
			a^{(i)}_{kl}(g)I_{d_i} \text{ if } j=i
		\end{cases}
	\end{equation}
	where $d_i=dim(V_i).$ ,
	
	Thus $\kappa_{ji}=0$ if $j\neq i$ and $\kappa_{ii}(g)=A^{(i)}(g)\otimes I_{d_i},$ where $A^{(i)}(g)=(a_{kl}^{(i)}(g))\in M_{n_i\times m_i}(\C)$ 
\end{proof}
\begin{remark}
	\begin{enumerate}
	\item In the special case, when the input and output spaces are the same ($V=W$), then $\kappa(g)$ is block-diagonal, i.e. we have the decomposition
	\begin{equation}
		\kappa(g)=\bigoplus_{i\in I} \kappa_{ii}=\bigoplus_{i\in I} A^{(i)}\otimes I_{d_i},
	\end{equation} where $A^{(i)}\in M_{m_i^2}(\C).$
	\item When $m_i=n_i=1$, then $A^{(i)}(g)=\lambda^{(i)}(g)\in C$, so $\kappa_{ii}(g)=\lambda^{(i)}(g)I_{d_i}$ is a scalar matrix.
\end{enumerate}
\end{remark}
\textbf{Significance of the decomposition theorems in G-CNN Architecture}

The decomposition theorems \ref{dec1}, \ref{red} provides a rigorous algebraic structure underlying the design of group-equivariant convolutional layers. Specifically, it asserts that any equivariant filter \( \psi: V \to W \)  can be reduced to block maps acting between isotypic components of the group representations.
This result has the following critical implications for model construction:

1. Structured Filter Design: The theorem \ref{red} restricts convolutional kernels to act only within matching irreducible subrepresentations. That is, if the input contains components transforming under an irreducible representation \( V_i \), then the output can only inherit information through matching components. This allows the design of convolution kernels to focus on specific channels aligned with each representation.

2. Efficient Parameterization: Rather than learning an arbitrary linear map between high-dimensional feature spaces (using $|G|\times dim(V)\times dim(W))=|G|\times \sum_{i\in I}m_id_i\sum_{j\in J}n_jd_j$ parameters for the filter $\psi$, the theorem \ref{red} allows the model to be parameterized in terms of smaller filters (namely, using only $|G|\times \sum_{i\in I\cap J}m_in_i$ parameters due to the formula \eqref{dec5}).  This reduces redundancy and leads to more efficient learning.

3. Representation-aware Feature Propagation: In the case of \( G = \mathrm{UT}_3(\mathbb{F}_3) \), which contains both 1-dimensional and 3-dimensional irreducible representations, the theorem ensures that the flow of information across layers respects the symmetry-induced decomposition. Each type of feature (e.g., local, global, structured) can be associated with different representations.

4. Robustness to Symmetric Deformations: By aligning filters to act within irreducible components, the model becomes inherently robust to group-preserving deformations, for example, handwritten input undergoing triangular shears or affine nilpotent distortions will still yield consistent activations in the corresponding channels.

In summary, the theorem provides both a mathematical guarantee and a practical guideline for designing interpretable, efficient, and symmetry-aware architectures when the underlying data domain admits a non-trivial group action such as \( \mathrm{UT}_3(\mathbb{F}_3) \).

\subsection{Preservation of Nested Subspaces under  $G$-Equivariant Convolution}
\begin{theorem}
	
	Let $G$ be a group and let  $(\rho_1,V),(\rho_2,W)$ be two representation of $G$. Suppose that there is a nested sequence of $G$-invariant subspaces of $V$:
	\[
	0 = V_0 \subset V_1 \subset V_2 \subset \cdots \subset V_r = V
	\]
	
	Let \( \mathcal{F}_{V_i} := \{ f: G \to V_i \} \) denote the corresponding subspace of feature maps.
	
	Let \( \Psi: \mathcal{F}_V \to \mathcal{F}_W \) be a \( G \)-linear map.  
	Then there exists a nested sequence of $G$-invariant subspaces of $W$: \[0=W_0\subset W_1\subset W_2\subset...\subset W_r=W\] such that
	\[
	\Psi(\mathcal{F}_{V_i}) \subseteq \mathcal{F}_{W_i}
	\]
\end{theorem}
\begin{proof}
Recall that the action of $G$ on $\mathcal{F}_V$ is defined by:
\begin{equation}
	\tilde{\rho}_1(g)f(h)=\rho_1(g)f(g^{-1}h)
\end{equation} for all $g,h\in G.$

If $f\in \mathcal{F}_{V_i}$, then $	\tilde{\rho}_1(g)f(h)=\rho_1(g)f(g^{-1}h)\in V_i$ since $V_i$ is $G$-invariant.

Now let $W_i:=\{(\Psi f(h))|f\in \mathcal{F}_{V_i},h\in G\},$ where $i=0,1,...,r$

Since $V_i\subset V_{i+1}$, we have $\mathcal{F}_{V_i}\subset \mathcal{F}_{V_{i+1}}$, so $W_i\subset W_{i+1}$, where $i=0,1,...,r-1.$

Now we show that $W_i$ is $G$-invariant. In fact, since $\Psi$ is $G$-equivariant, we have $\tilde{\rho}_2(g)\Psi f(h)=\Psi (\tilde{\rho}_1(g)f(h))=\Psi f(g^{-1}h)\in W_i.$ 

By the definition of $W_i$, for all $f\in \mathcal{F}_{V_i}, h\in G$ we have: $(\Psi f)(h)\in W_i$. So $\Psi f\in \mathcal{F}_{W_i}$ for all $f\in \mathcal{F}_{V_i}.$ Therefore $\Psi (\mathcal{F}_{V_i})\subseteq \mathcal{F}_{W_i}.$ 
\end{proof}
 This result shows that equivariant convolutional layers preserve chains of invariant subspaces. In hierarchical models, such structure often corresponds to semantic or geometric abstraction levels. Preserving this chain ensures consistency in feature refinement across layers.
\subsection{Representations of $UT_3(\F_3)$}
\subsubsection{Structure of $UT_3(\F_3)$}
Let $G = UT_3(\mathbb{F}_3)$ be the group of $3\times3$ upper triangular matrices with 1s on the diagonal and entries from $\mathbb{F}_3=\{0,1,2\}$ (the field with 3 elements). Explicitly, each element of $UT_3(\F_3)$ has the form:
\[
g(a,b,c) = \begin{pmatrix}
	1 & a & b \\
	0 & 1 & c \\
	0 & 0 & 1
\end{pmatrix}, \quad a,b,c \in \mathbb{F}_3.
\]
Hence, $|G| = 3^3 = 27$. Furthermore, we have:

\begin{proposition}
	$UT_3(\F_3)$ is a nilpotent group of class 2.
\end{proposition}
\begin{proof}

	Observe that given a matrix $A = \begin{pmatrix}
		1&a&b\\0&1&c\\0&0&1
	\end{pmatrix}$, we have $A^{-1}= \begin{pmatrix}
		1&-a&-b+ac\\0&1&-c\\0&0&1
	\end{pmatrix}$. \\
	Let $G_0 = G=UT_3(\F_3)$
	We have:
	\begin{equation}
		G_1 = [G_0, G_0] =\langle ABA^{-1}B^{-1}|A,B\in G\rangle 
	\end{equation}
	A direct computation gives
	\begin{multline}
		ABA^{-1}B^{-1}= \begin{pmatrix} 1&a_{1}&b_{1}\\0&1&c_{1}\\0&0&1 \end{pmatrix} 
		\begin{pmatrix} 1&a_{2}&b_{2}\\0&1&c_{2}\\0&0&1 \end{pmatrix}  
		\begin{pmatrix} 1&-a_{1}&-b_{1}+a_{1}c_{1}\\0&1&-c_{1}\\0&0&1 \end{pmatrix}
		\begin{pmatrix} 1&-a_{2}&-b_{2}+a_{2}c_{2}\\0&1&-c_{2}\\0&0&1 \end{pmatrix} \\
		= \begin{pmatrix} 1&0 & a_{1}c_{2} - a_{2}c_{1}\\0&1&0\\0&0&1 \end{pmatrix}\in Z(G)\implies G_1=Z(G) \end{multline}
	Thus $ G_2 = [G, G_1]=[G,Z(G)]=\{I_3\} $
	
	By definition, $UT_3(\F_3)$ is a nilpotent group of class 2.   
\end{proof}
The nilpotency of $UT_3(\F_3)$ has several important implications in designing equivariant models such as G-CNNs:

\begin{itemize}
	\item \textbf{Hierarchical Decomposition:} Nilpotent groups admit a central series $G = G_0 \triangleright G_1 \triangleright G_2 = \{e\}$, enabling structured analysis and learning from global to local symmetries.
	\item \textbf{Abundance of 1D Representations:} The large abelian quotient $G/[G,G] \cong \mathbb{F}_3^2$ results in numerous one-dimensional representations, useful for building interpretable and computationally efficient modules.
	\item \textbf{Simplified Convolution:} The group convolution over a nilpotent group reduces to algebraically tractable forms, as the center and commutator subgroups simplify the integration (or summation) structure.
	\item \textbf{Natural Equivariance for Structured Data:} Data types such as symbolic sequences, handwritten formulas, or linguistic constructs often exhibit symmetries well-modeled by nilpotent actions. Thus, using $\mathrm{UT}_3(\mathbb{F}_3)$ aligns model symmetry with data symmetry.
\end{itemize}

These aspects make nilpotent groups particularly attractive in theory-driven neural architectures, especially where symbolic structure and hierarchical reasoning are important.

\subsubsection{Conjugacy classes of $UT_3(\F_3)$}
A direct computation shows that
$UT_3(\F_3)$ has 11 distinct conjugacy classes, as follow:
\begin{itemize}
	\item Three classes of size 1: $C_1=\left \{\begin{pmatrix} 1&0&0\\0&1&0\\0&0&1\end{pmatrix} \right \},
	C_2=\left \{\begin{pmatrix} 1&0&1\\0&1&0\\0&0&1\end{pmatrix} \right \}, 
C_3=	\left \{\begin{pmatrix} 1&0&2\\0&1&0\\0&0&1\end{pmatrix} \right \}$.
	\item Eight classes of size 3:
	\begin{itemize}
		
		\item $C_4=\{\begin{pmatrix} 1&0&0\\0&1&1\\0&0&1\end{pmatrix} ,
		\begin{pmatrix} 1&0&1\\0&1&1\\0&0&1\end{pmatrix} , 
		\begin{pmatrix} 1&0&2\\0&1&1\\0&0&1\end{pmatrix}\}$
		
		\item $C_5=\{\begin{pmatrix} 1&0&0\\0&1&2\\0&0&1\end{pmatrix} ,
		\begin{pmatrix} 1&0&1\\0&1&2\\0&0&1\end{pmatrix} , 
		\begin{pmatrix} 1&0&2\\0&1&2\\0&0&1\end{pmatrix}\}$
		%1 in a_12, 0 in a_22
		\item $C_6=\{\begin{pmatrix} 1&1&0\\0&1&0\\0&0&1\end{pmatrix} ,
		\begin{pmatrix} 1&1&1\\0&1&0\\0&0&1\end{pmatrix} , 
		\begin{pmatrix} 1&1&2\\0&1&0\\0&0&1\end{pmatrix}\}$
		
		\item $C_7=\{\begin{pmatrix} 1&1&0\\0&1&1\\0&0&1\end{pmatrix} ,
		\begin{pmatrix} 1&1&1\\0&1&1\\0&0&1\end{pmatrix} , 
		\begin{pmatrix} 1&1&2\\0&1&1\\0&0&1\end{pmatrix}\}$
		%1 in a_12, 2 in a_22
		\item $C_8=\{\begin{pmatrix} 1&1&0\\0&1&2\\0&0&1\end{pmatrix} ,
		\begin{pmatrix} 1&1&1\\0&1&2\\0&0&1\end{pmatrix} , 
		\begin{pmatrix} 1&1&2\\0&1&2\\0&0&1\end{pmatrix}\}$
		%2 in a_12, 0 in a_22
		\item $C_9=\{\begin{pmatrix} 1&2&0\\0&1&0\\0&0&1\end{pmatrix} ,
		\begin{pmatrix} 1&2&1\\0&1&0\\0&0&1\end{pmatrix} , 
		\begin{pmatrix} 1&2&2\\0&1&0\\0&0&1\end{pmatrix} \}$
		%2 in a_12, 1 in a_22
		\item $C_{10}=\{\begin{pmatrix} 1&2&0\\0&1&1\\0&0&1\end{pmatrix} ,
		\begin{pmatrix} 1&2&1\\0&1&1\\0&0&1\end{pmatrix} , 
		\begin{pmatrix} 1&2&2\\0&1&1\\0&0&1\end{pmatrix}\} $
		%2 in a_12, 2 in a_22
		\item $C_{11}=\{\begin{pmatrix} 1&2&0\\0&1&2\\0&0&1\end{pmatrix} ,
		\begin{pmatrix} 1&2&1\\0&1&2\\0&0&1\end{pmatrix} , 
		\begin{pmatrix} 1&2&2\\0&1&2\\0&0&1\end{pmatrix} \}$
	\end{itemize}
\end{itemize}
Therefore $UT_3(\F_3)$ has 11 nonisomorphic irreducible representations due to Theorem \ref{nirr}.
\subsubsection{Irreducible representations of $UT_3(\F_3)$ }\label{irrep}
\begin{enumerate}
	\item[a)]\textbf{One dimensional irreducible representations of $UT_3(\F_3)$}
	
	We begin with the following definition:
	\begin{definition}
		Let $G$ be a group. The \textbf{abelianization} of $G$, denoted by $G^{ab}$ is defined by
		\begin{equation}
			G^{ab}:=G/[G,G].
		\end{equation}
	\end{definition}
		To construct the one-dimensional irreducible representations of the group \( G=UT_3(\mathbb{F}_3) \), we begin with the classical result in representation theory:
	
	\begin{theorem}
		Let \( G \) be a finite group. Then there is a canonical group isomorphism:
		\[
		\mathrm{Hom}(G, \mathbb{C}^*) \cong \mathrm{Hom}(G^{ab}, \mathbb{C}^*),
		\]
	
	\end{theorem}
\begin{proof}
	Consider the map $T: \mathrm{Hom}(G, \mathbb{C}^*) \to \mathrm{Hom}(G^{\mathrm{ab}}, \mathbb{C}^*)$ defined by $T(\varphi)([g])=\varphi(g)$, where $\varphi\in \mathrm{Hom}(G, \mathbb{C}^*), [g]$ is the left coset of $[G,G]$ in $G$. 
	
	Remark that  \begin{equation}\label{hom1}
		T(\varphi)([g][h])=T(\varphi)([gh])=\varphi(gh)=\varphi(g)\varphi(h)=T(\varphi)([g])T(\varphi)([h]).\end{equation}
	
	At first, we verify that $T$ is well-defined. In fact, if  $[g]=[g']$, then  $g^{-1}g'\in [G,G]$, i.e. $g^{-1}g'=\prod_i[a_i,b_i]$, where $a_i,b_i\in G$. 
	
	Then \begin{multline}
		T(\varphi)([g]^{-1}[g'])=\varphi(g^{-1}g')=\prod_{i}\varphi([a_i,b_i])=\prod_i\varphi(a_i)\varphi(b_i)\varphi_(a_i^{-1})\varphi(b_i^{-1})=\\=\varphi(a_i)\varphi(a_i^{-1})\varphi(b_i)\varphi(b_i^{-1})=1\end{multline} since $\C^*$ is abelian. 
	
Therefore by \eqref{hom1} we have $T(\varphi)([g])T(\varphi)([g']^{-1})=1\implies T(\varphi)([g])=T(\varphi)([g']).$ This means $T$ is well-defined. 

By similar arguments, $T$ injective. 

Furthermore, $T(\varphi_1\varphi_2)([g])=(\varphi_1\varphi_2)(g)=\varphi_1(g)\varphi_2(g)=T(\varphi_1)[g]T(\varphi_2)(g)$ for all $g\in G, \varphi_1,\varphi_2\in Hom(G,\C^*)$. 

Thus $T(\varphi_1\varphi_2)=T(\varphi_1)T(\varphi_2)$ and so $T$ is a group homomorphism. 

In summary, $T$ is a group isomorphism. 
\end{proof}
	In our case, the abelianization of \( G \) is:
	\begin{proposition}
		\[
		G^{\mathrm{ab}} \cong \F_3\times \F_3,
		\] as groups.
	\end{proposition}
	\begin{proof}
		$G^{ab}=G/[G,G]$ and two elements $g(a,b,c)$ and $g'(a',b',c')$ are in the same left coset of $[G,G]$ iff $g^{-1}g'\in [G,G].$
		
		On the other hand
		\begin{multline} \label{hom2}
			g^{-1}g'=\begin{pmatrix}
				1&a'-a&b'-a'c-b+ac
				\\0&1&c'-c\\0&0&1
			\end{pmatrix}\in [G,G]=Z(G) \Leftrightarrow  a=a',c=c'. 
		\end{multline} . 
		
		Now we construct a map $\phi: G^{ab}\to \F_3\times \F_3$ by $\phi([g(a,b,c)])
(a,c)$ and show that it is an isomorphism.

At first, following \eqref{hom2} $\phi$ is well-defined and injective. 

Moreover, \begin{multline}\phi ([g(a,b,c)][g'(a',b',c')])=\phi([\begin{pmatrix}
	1&a&b\\0&1&c\\0&0&1
\end{pmatrix}\begin{pmatrix}
1&a'&b'\\
0&1&c'\\
0&0&1
\end{pmatrix}])=\phi([\begin{pmatrix}
1&a+a'&b'+ac'+b\\
0&1&c+c'\\
0&0&1
\end{pmatrix}])=\\=(a+a',c+c')=\phi([g(a,b,c)])+\phi([g'(a',b',c')]).\end{multline} 

So $\phi$ is a homomorphism.	

Thus $\phi$ is an isomorphism.
\end{proof}
	
Therefore we can identify the element $[g(a,b,c)]$ with $(a,c)\in \F_3\times \F_3$	and the goal is to find all one dimensional irreducible representations of $\F_3\times \F_3$.  
	
	Every complex representation of $\F_3\times \F_3$ is a group homomorphism $\rho: \F_3\times \F_3\to \C^*$ and $\F_3\times \F_3=\langle (1,0),(0,1)\rangle$, so $\rho$ is uniquely determined by $\rho((1,0))$ and $\rho((0,1))$.  Namely, for all $(a,b)\in \Z_3\times \Z_3$ we have:
	\begin{equation}\label{rep1}
		\rho((a,c))=\rho(a(1,0)+c(0,1))=\rho((1,0))^a\rho((0,1))^c.
	\end{equation}

	On the other hand, $\rho(3(1,0))=\rho((0,0))=1=\rho((1,0))^3$, so $\rho((1,0))$ must be a third root of unity and therefore has the form $\rho((1,0))=\omega^r,$ where $\omega=e^{2\pi i/3}$ (the primitive third root of unity) and $r\in \F_3$.  Similarly, $\rho((0,1))=\omega^s, $ where $s\in \F_3.$.
	
	Substituiting $\rho((1,0))$ and $\rho((0,1))$ gives:
	\begin{equation}
		\rho((a,c))=\omega^{ra}\omega^{sc}=\omega^{ra+sc}.
	\end{equation}
	
	Thus $UT_3(\F_3)$ has $9$ nonisomorphic one dimensional representations $\rho_{r,s}$ corresponding to $9$ pairs $(r,s)\in \Z_3^2$ and determined by 
	\begin{equation}
		\rho_{r,s}(g(a,b,c))=\omega^{ra}\omega^{sc}=\omega^{ra+sc}.
	\end{equation} for all $a,c\in \F_3.$
	
		The character values of one-dimensional irreducible representations are clear: 
	\begin{equation}\label{char}
		\chi_{\rho_{r,s}}((a,c))=\omega^{ra+sc}.
	\end{equation}
	for all $(a,c)\in \F_3\times \F_3.$
	\item[b)] Three dimensional irreducible representations of $UT_3(\F_3)$ 
	
	Besides the 9 one-dimensional irreducible representations of \( G = \mathrm{UT}_3(\mathbb{F}_3) \), there exist exactly two irreducible representations of degree 3. These arise via induced representations from the center \( Z \) of \( G \).

	The center is:
	\[
	Z = \left\{ \begin{pmatrix} 1 & 0 & b \\ 0 & 1 & 0 \\ 0 & 0 & 1 \end{pmatrix} \middle| b \in \mathbb{F}_3 \right\} \cong \mathbb{F}_3.
	\]
	Choose a nontrivial character of the center \( \psi_k: Z \to \mathbb{C}^* \) defined by:
	\[
	\tau_k\left(\begin{pmatrix} 1 & 0 & b \\ 0 & 1 & 0 \\ 0 & 0 & 1 \end{pmatrix}\right) = \omega^{kb}, \quad k = 1,2.
	\]
	Then extend it to the representation $\tau_k$ of the subgroup \( H = \left\{ \begin{pmatrix} 1 & 0 & b \\ 0 & 1 & c \\ 0 & 0 & 1 \end{pmatrix} \right\} \cong \mathbb{Z}_3^2 \), which is abelian.

	Let \( \rho_k = \mathrm{Ind}_H^G(\tau_k) \). By Remark \ref{ind6} these representations have dimension $[G:H]=\frac{27}{9}=3$ and later by using the character table we shall prove that these induced representations are irreducible and nonisomorphic. 
	
	Now we shall give the explicit matrix form of these 3-dimensional representations. 
	
	Let \begin{equation}\label{ind1}V=\{\phi: G\to \C^*|f(hx)=\tau_k(h)\phi(x)\quad \forall h\in H, x\in G\}.\end{equation}
	
	Following Remark \ref{ind6}, each element $\phi\in V$ is uniquely determined by $\{\phi(r)\}$, where $r$ are representatives of the right cosets of $H$. 
	\begin{proposition}
		Two elements $g(a,b,c)$ and $g'(a',b',c')$ of $G$ are in the same right coset of $H$ if and only if $a=a'.$
	\end{proposition}
	\begin{proof}
		$g$ and $g'$ are in the same right coset of $H$ if and only if $g'=hg$ for some $h\in H$ or equivalently, $g'g^{-1}\in H$.
		
		We have $gg'^{-1}=\begin{pmatrix}
			1&a'&b'\\
			0&1&c'\\
			0&0&1
		\end{pmatrix}\begin{pmatrix}
		1&-a&-b+ac\\
		0&1&-c\\
		0&0&1
		\end{pmatrix}=\begin{pmatrix}
		1&-a+a'&-b+ac-a'c+b'\\
		0&1&-c+c'\\
		0&0&1
		\end{pmatrix}\in H$ if and only if $a=a'$. 
	\end{proof}
	
	Therefore there are three right cosets of $H$ with representatives:$g(0,0,0),g(1,0,0), g(2,0,0)$ and each $\phi\in V$ is uniquely determined by the images of these representatives. 
	
	Let $\rho_k=Ind_H^G(\tau_k)$. By the definition of the induced representation we have:
	\begin{equation}
		(\rho_k(g(a,b,c))\phi)(g(0,0,0))=\phi(g(0,0,0)g(a,b,c))=\phi(g(a,b,c)).
	\end{equation}
	On the other hand, 
	\begin{equation}
		g(a,b,c)=\begin{pmatrix}
			1&a&b\\
			0&1&c\\
			0&0&1
		\end{pmatrix}=\begin{pmatrix}
	1&0&b\\
	0&1&c\\
	0&0&1
		\end{pmatrix}\begin{pmatrix}
		1&a&0\\0&1&0\\0&0&1
		\end{pmatrix}=g(0,b,c)g(a,0,0).
	\end{equation}
	So by \eqref{ind1} we have
	\begin{equation}\label{ind2}
	(\rho_k(g(a,b,c))\phi)(g(0,0,0))=\phi(g(a,b,c))=\rho_k(g(0,b,c))\phi(g(a,0,0))=\omega^{kb}\phi(g(a,0,0)).
	\end{equation}
	Similarly, we can find
	\begin{align}
		(\rho_k(g(a,b,c))\phi)(g(1,0,0))=\omega^{k(b+c)}\phi(g(a+1,0,0))\label{ind3}\\
		(\rho_k(g(a,b,c))\phi)(g(2,0,0))=\omega^{k(b+2c)}\phi(g(a+2,0,0))\label{ind4}
	\end{align}
	From \eqref{ind2}, \eqref{ind3},\eqref{ind4} we deduce that
	\begin{equation}\label{irr2}
		(\rho_k(g(a,b,c))\phi)(n,0,0)=\omega^{k(b+nc)}\phi(g(a+n,0,0))
	\end{equation} where $n\in \F_3,k\in \{1,2\}$

	\subsubsection{Character Table}
	
	Using formulas \eqref{char} and \eqref{irr2} we can find the characters of all representations described in the subsection \ref{irrep} as in the following table:
	\begin{table}[h!]\label{tab1}
		\centering
	\caption{Character Table of $\mathrm{UT}_3(\mathbb{F}_3)$}
	\begin{center}
		\begin{tabular}{|c|c|c|c|c|c|c|c|c|c|c|c|c|c}
			
			\hline
			Class & Representative & $\chi_{\rho_{0,0}}$ & $\chi_{\rho_{0,1}}$ & $\chi_{\rho_{0,2}}$&$\chi_{\rho_{1,0}}$&$\chi_{\rho_{1,1}}$&$\chi_{\rho_{1,2}}$&$\chi_{\rho_{2,0}}$&$\chi_{\rho_{2,1}}$&$\chi_{\rho_{2,2}}$&$\chi_{\rho_{1}}$&$\chi_{\rho_{2}}$\\\hline
			$C_1$& $I_3$ & $1$ & $1$ & $1$ &$1$&$1$&$1$&$1$&$1$&$1$& $3$&$3$\\\hline
			$C_2$  & $g(0,1,0)$ & $1$&$1$&$1$&$1$&$1$&$1$&$1$&$1$&$1$ & $3\omega$ & $3\omega^2$ \\\hline
			$C_3$ & $g(0,2,0)$ & $1$&$1$&$1$&$1$&$1$&$1$&$1$&$1$&$1$ & $3\omega^2$ & $3\omega$ \\\hline
			$C_4$  & $g(0,0,1)$ & $1$ & $\omega$&$\omega^2$& $1$ & $\omega$&$\omega^2$& $1$ & $\omega$&$\omega ^2$ & $0$ &$0$\\\hline
			$C_5$  & $g(0,0,2)$ & $1$ & $\omega^2$&$\omega$& $1$ & $\omega^2$&$\omega$& $1$ & $\omega^2$&$\omega$ & $0$ &$0$ \\\hline
			$C_6$  & $g(1,0,0)$ & $1$ & $1$&$1$&$\omega$& $\omega$&$\omega$&$\omega^2$&$\omega^2$&$\omega^2$&$0$& $0$ \\\hline
			$C_7$ & $g(1,0,1)$ & $1$ & $\omega$ & $\omega^2$&$\omega$&$\omega^2$&$1$&$\omega^2$&$1$&$\omega$&$0$&$0$ \\\hline
			$C_8$  & $g(1,0,2)$ & $1$ & $\omega^2$ &$\omega$&$\omega$&$1$&$\omega^2$&$\omega^2$&$\omega$&$1$& $0$&$0$ \\\hline
			$C_9$ & $g(2,0,0)$ & $1$&$1$&$1$&$\omega^2$&$\omega^2$&$\omega^2$ & $\omega$&$\omega$&$\omega$ & $0$&$0$ \\\hline
			$C_{10}$&$g(2,0,1)$&$1$&$\omega$&$\omega^2$&$\omega^2$&$1$&$\omega$&$\omega$&$\omega^2$&$1$&$0$&$0$\\\hline
			$C_{11}$&$g(2,0,2)$&$1$&$\omega^2$&$\omega$&$\omega^2$&$\omega$&$1$&$\omega$&$1$&$\omega^2$&$0$&$0$\\\hline
		\end{tabular}
	\end{center}
\end{table}

Now we use the characters of $\rho_1,\rho_2$ given in the previous table and Theorem \ref{ort1} to verify that they are irreducible:
\begin{proposition}
	The representations $\rho_1,\rho_2$ described by formula \eqref{irr2} are irreducible. 
\end{proposition}
	\begin{proof}
	 We have $\langle \chi_{\rho_1},\chi_{\rho_1}\rangle=\frac{1}{27}(3\times 3+3\omega\times 3\bar{\omega}+3\omega^2\times 3\overline{\omega^2})=\frac{27}{27}=1.$
		
		So by Theorem \ref{ort1} $\rho_1$ is irreducible. Similarly, $\rho_2$ is irreducible. 
	\end{proof}
	
	The knowledge of the complete set of irreducible representations and their characters allows the construction of $G$-equivariant feature spaces and convolution kernels. These tools are fundamental for the steerable architectures discussed later in this paper.

	It is necessary to mention that although existing software packages such as GAP can compute the character table of $UT_3(\F_3)$, 
and list its representations, but
	they do not explicitly construct these representations in matrix form aligned with group convolutional learning and do not connect them to filter decomposition or feature design in CNNs.
	
	Prior mathematical studies (e.g., on Heisenberg groups) cover general structures, but do not integrate the results into a machine learning framework.
	
	Our work constructs explicit matrix forms of irreducible representations, analyzes their action on feature spaces, and proves decomposition theorems (Section 3) essential for building G-CNN layers with theoretical guarantees.
	
	This is, to our knowledge, the first complete integration of the representation theory of 
	$UT_3(\F_3)$
into the architecture design of equivariant deep learning models.

\section{Architecture: G-CNN with Transformer Decoder Based on Irreducible Representations}

In this section, we present a complete architecture that combines a group-equivariant convolutional neural network (G-CNN) based on the irreducible representations (irreps for short) of the nilpotent group $G = UT_3(\mathbb{F}_3)$ with a Transformer decoder to perform handwritten mathematical expression recognition. This design avoids the use of the regular representation and instead directly structures the network according to irreps.

Each layer of the G-CNN consists of blocks acting independently on each irrep. Let $f_j^{(\ell)}$ be the feature map of irrep $j$ at layer $\ell$.

\subsection{Feature Extraction Using Irrep-Wise G-CNN}
We define the following modules:

\paragraph{Irrep Feature Map Format:}
Each scalar irrep $\chi_i$ has a feature map of shape $\mathbb{R}^{C_i \times H \times W}$. Each 3-dimensional irrep $\rho_k$ has a feature map $f_k(x) \in \mathbb{R}^{C_k \times 3 \times H \times W}$.

\paragraph{Equivariant Convolution:}
Given $f_j^{(\ell)}(x)$ and $\rho_j$, for each pair $(i,j)$ with $\rho_i \cong \rho_j$, we define:
\[
\Psi_{ij}^{(\ell)}(g) = A_{ij}^{(\ell)}(g) \otimes I_{d_i}, \quad A_{ij}^{(\ell)}(g) \in \mathbb{C}^{C_i \times C_j},
\]
and perform convolution:
\[
[f_i^{(\ell+1)}](x) = \sum_{j:\, \rho_j \cong \rho_i} \sum_{y \in G} A_{ij}^{(\ell)}(x^{-1}y) f_j^{(\ell)}(y).
\]

\paragraph{Nonlinearity:}
\begin{itemize}
	\item For 1-dimensional  irreps: apply standard nonlinearity $\sigma(f) = \mathrm{ReLU}(f)$.
	\item For 3-dimensional irreps: apply norm nonlinearity:
	\[
	\sigma(f(x)) = \frac{f(x)}{\|f(x)\|} \cdot \phi(\|f(x)\|)
	\]
	where $\phi$ is a scalar activation function.
\end{itemize}

\paragraph{Pooling Layer:}
We apply spatial pooling independently within each irrep branch:
\[
f_i^{(\ell)} \mapsto \mathrm{Pool}(f_i^{(\ell)})
\]

\paragraph{Concatenation and Projection:}
After $L$ layers, we concatenate all irreps:
\[
X = [f_1, \ldots, f_9, f_{10}^{(1)}, f_{11}^{(2)}] \in \mathbb{R}^{T \times C}, \quad C = \sum_{i=1}^9 C_i + \sum_{k=10}^{11} 3C_k
\]

\subsection{Transformer Decoder for Sequence Modeling}
Let $X = [x_1, x_2, \ldots, x_T]$, $x_t \in \mathbb{R}^C$ be the output from G-CNN. The decoder predicts a LaTeX sequence $y = [y_1, y_2, \ldots, y_L]$.

\paragraph{Positional Encoding:}
\[
\tilde{x}_t = x_t + p_t, \quad p_t = \text{PosEnc}(t)
\]

\paragraph{Decoder Layers:}
Each layer consists of self-attention, cross-attention, and feed-forward network:
\begin{align*}
	Q_t &= W_q y_{t-1}, \quad K_t = W_k \tilde{x}_t, \quad V_t = W_v \tilde{x}_t \\
	\text{Attention}(Q, K, V) &= \mathrm{softmax}\left(\frac{QK^\top}{\sqrt{d_k}}\right)V
\end{align*}

\paragraph{Output Projection:}
\[
\hat{y}_t = \mathrm{Softmax}(W_o h_t + b_o)
\]

\subsection{Loss Function}
We use cross-entropy with teacher forcing:
\[
\mathcal{L}(y, \hat{y}) = -\sum_{t=1}^L \log P(y_t | y_{<t}, X)
\]

\subsection{Remarks}
This architecture respects the group symmetry in the convolutional stage and achieves strong symbolic decoding capability via the transformer decoder. The direct use of irreducible components eliminates the need for a change-of-basis matrix.

%This finite nilpotent group of order 27 admits a rich structure of representations, allowing us to design a convolutional network that is equivariant under its group action. Our approach decomposes the regular representation into irreducible components, applies equivariant convolutions at the level of each component, and finally integrates the learned features for classification.

%\subsection{Decomposition of the Regular Representation}

\subsection{A concrete example on \(UT_{3}(\mathbb{F}_{3})\)}

We write elements of 
\(\displaystyle G = UT_{3}(\mathbb{F}_{3})\) 
as 
\[
g(a,b,c)
=\begin{pmatrix}
	1 & a & c\\
	0 & 1 & b\\
	0 & 0 & 1
\end{pmatrix},\quad a,b,c\in\{0,1,2\}.
\]  

One of the nine 1-dimensional irreducible characters is
\[
\chi_{\rho_{(1,0)}}\bigl(g(a,b,c)\bigr)
=\exp\!\bigl(\tfrac{2\pi i}{3}\,a\bigr).
\]
Let us define a constant feature map \(f: G\to\mathbb{C}\), 
\[
f\bigl(g(a,b,c)\bigr)=1\quad\forall\,g\in G.
\]
We choose the corresponding scalar filter (kernel) 
\(\psi=\chi_{\rho_{(1,0)}}\).  The equivariant convolution at the identity is
\[
(f \star  \psi)(e)
=\sum_{h\in G}f(h)\,\psi\bigl(e^{-1}h\bigr)
=\sum_{a,b,c\in\{0,1,2\}} 
\exp\!\Bigl(\tfrac{2\pi i}{3}\,a\Bigr).
\]
Since for fixed \(b,c\) the inner sum 
\(\sum_{a=0}^2e^{2\pi i a/3}=0\), the total vanishes:
\[
(f \star \psi)(e)=0.
\]

By contrast, if we take the \emph{trivial} character 
\(\chi_{\mathrm{triv}}(g)=1\), then
\[
(f * \chi_{\mathrm{triv}})(e)
=\sum_{h\in G}1\cdot1
=27,
\]
which agrees with projecting onto the trivial isotypic subspace.

\medskip

\noindent\textbf{Interpretation.}  
\begin{enumerate}
\item Using a non-trivial 1-dimensional character \(\chi_{(1,0)}\) “picks out” that irreducible component, and the sum vanishes by orthogonality of characters.  

\item Using the trivial character projects onto the invariant subspace, summing to \(|G|=27\).  

This real example on \(UT_3(\mathbb{F}_3)\) vividly shows how equivariant convolution
\[
(f * \psi)^{(\rho)}(e)
=\sum_{h\in G}f(h)\,\psi\bigl(h^{-1}\bigr)
\]
implements the orthogonal projection onto each irreducible subrepresentation.  

\end{enumerate}

\section{Conclusion and Outlook}

In this paper, we classified all irreducible complex representations of $UT_3(\mathbb{F}_3)$ and computed its complete character table. We proved a decomposition theorem for $G$-modules and established how equivariant operators must respect the isotypic decomposition.

Our results are purely algebraic yet carry implications for the theory of symmetry in functional structures. For instance, when considering a neural network layer as a linear map between feature spaces carrying $G$-actions, our theorems dictate that these maps can be partitioned into smaller   filters which lie in $\mathrm{Hom}_G(V_i, V_j)$ (Theorem \ref{red}) and thus be highly structured.

This approach contributes to the understanding of symmetry-preserving operators in abstract settings, such as:
\begin{itemize}
	\item Designing mathematically sound equivariant transformations,
	\item Classifying symmetry types in data spaces,
	\item Predicting the algebraic constraints on learning architectures.
\end{itemize}

Looking forward, the same representation‐theoretic paradigm can be naturally generalized to larger upper‐triangular groups
\(UT_n(\mathbb{F}_q)\) for \(n>3\) or arbitrary finite fields \(\mathbb{F}_q\).  Such groups have classical decompositions into one‐dimensional and higher‐dimensional irreps (via Kirillov’s orbit method \cite{Kirillov}), so one can likewise build G-CNN layers by placing scalar filters on each isotypic block.  Although the number of irreps and the block‐sizes grow (e.g.\ \(|UT_n(\F_q)=q^{n(n-1)/2}\)), the same Schur‐lemma reduction ensures that each layer remains block‐diagonal and therefore computationally tractable. Beyond unitriangular groups, other classes of finite nilpotent (or even solvable) groups such as the Heisenberg groups over \(\mathbb{F}_q\) or direct products of \(p\)-groups admit similar representation decompositions.  We therefore anticipate that our framework can be extended to these settings, opening the door to G-CNNs tailored to a wide range of algebraic symmetries arising in structured vision and symbolic reasoning tasks.

Future work will explore these generalizations in detail, examining both the practical trade-offs in filter count and computational cost, and the gains in equivariance and data efficiency on larger or more complex symmetry groups.

\vskip0.3cm\noindent\textbf{Acknowledgements.} The authors are grateful to the referees for their careful reading of the manuscript and their useful comments.

%%%If you use bibitex, please remove percentage sign and delete the bottom one
%\bibliographystyle{amsplain}\addcontentsline{toc}{section}{References}
%\bibliography{References}

\end{enumerate}
\end{document}